\documentclass[11pt]{article}

\usepackage[right=1in, left=1in, top=1in, bottom=1in]{geometry}

\usepackage{graphicx} 
\usepackage{amsmath,enumitem,xcolor,amssymb,amsthm}
\usepackage[dvipsnames]{xcolor}

\usepackage{hyperref}
\hypersetup{
    colorlinks=true,
    linkcolor=blue,
    filecolor=magenta,      
    urlcolor=cyan,
    breaklinks=true,
    }
\usepackage{xurl}
\usepackage{doi}

\newtheorem{theorem}{Theorem}[section]
\newtheorem{prop}[theorem]{Proposition}
\newtheorem{lemma}[theorem]{Lemma}
\newtheorem{claim}[theorem]{Claim}

\newtheorem{cor}[theorem]{Corollary}
\newtheorem{obs}[theorem]{Observation}
\newtheorem{dfn}[theorem]{Definition}

\let\epsilon=\varepsilon

\begin{document}

\title{On Balance, To What Degree is Burr's Conjecture True?} 
\author{Shagnik Das\footnote{Department of Mathematics, National Taiwan University, Taipei 10617, Taiwan. Email: \texttt{shagnik@ntu.edu.tw}. Supported by NSTC Grant 113-2628-M-002-008-MY4.} 
\and 
Bruce Reed\footnote{Institute of Mathematics, Academia Sinica, Taipei 106319, Taiwan. Email: \texttt{bruce.al.reed@gmail.com}. Supported by NSTC Grant 112-2115-M-001-013-MY3.} 
\and 
Jozef Skokan\footnote{London School of Economics, Department of Mathematics, London, WC2A 2AE, UK; Email: \texttt{j.skokan@lse.ac.uk}.} }
\date{\today}

\maketitle
\begin{abstract}
    For many trees $T$, the Ramsey number of $T$, denoted by ${\mathcal R}(T)$, is determined by the sizes of the partition classes in its unique bipartition. In 1976, Burr proved that when $T$ has partition classes of size $t_1$ and $t_2$ with $t_1 \le t_2$, the Ramsey number is at least $\max(2t_2-1,2t_1+t_2-1)$, and conjectured that this is tight. While counterexamples have been found for some pairs $(t_1, t_2)$, a main focus of research on this problem has been determining ratios $t_2/t_1$ or bounds on the maximum degree of $T$ for which Burr's bound is either exactly or asymptotically tight.
    
    We essentially resolve these questions for lopsided trees. Specifically, we show that (a) there are counterexamples whenever $t_2 \ge 2t_1$, with the order of magnitude of the difference between the largest Ramsey numbers and Burr's bound being $\max \left( t_1^2/t_2, \sqrt{t_1} \right)$, and (b) for $t_2 \ge 500 t_1$, Burr's bound is tight when $\Delta(T) \le t_2 - t_1$, but is off by at least $C \log t_2$ (even when $t_2 \ge 2 t_1$) when $\Delta(T) \gtrsim t_2 - t_1$. In particular, this shows that Burr's bound need not hold for $t$-vertex trees $T$ with $\Delta(T) \approx t/3$.
\end{abstract}

\thispagestyle{empty}
\section{Introduction }

The Ramsey number of a tree $T$, ${\cal R}(T)$  is the minimum $n$ such that every two colouring of the edges of clique $K_n$ on $n$ vertices yields a monochromatic copy of $T$. This graph parameter has been studied intensely over the years, and varies greatly depending on the structure of the tree in question. For instance, it is known that the Ramsey number of the path $P_t$ on $t$ vertices is $\lfloor \frac{3t-2}{2} \rfloor$ (c.f.~\cite{GG}), while Harary~\cite{Ha} noted that the Ramsey number of the $(t-1)$-star $K_{1, t-1}$ (which has $t$ vertices) comes from colourings of $K_n$ where each monochromatic subgraph has maximum degree $t-2$, which implies ${\cal R}(K_{1,t-1}) = 2t-2$ when $t$ is even and ${\cal R}(K_{1,t-1}) = 2t-3$ when $t$ is odd. The large discrepancy between these two Ramsey numbers is explained by the fact that the partition classes in the unique bipartition of the star are as unbalanced as possible while those in a path are as balanced as possible. 

For a tree $T$ we let $I_1=I_1(T)$ and $I_2=I_2(T)$ be the two partition classes in its unique bipartition labelled so that $|I_1|\le|I_2|$. Trees with $|I_j|=t_j$, $j=1,2$,  are called {\it $(t_1,t_2)$-trees}. 
We define $$f(t_1, t_2) = \max \{ {\cal R}(T): T \text{ is a } (t_1, t_2)\text{-tree} \}.$$
In  1974, Burr \cite{Burr} conjectured that  the Ramsey number of a tree $T$ with $t$ nodes is determined solely by the balance between the two partition classes, i.e. that every $(t_1,t_2)$-tree $T$ satisfies $R(T)=f(t_1,t_2)$, where 
\begin{equation}\label{eq:burr}
f(t_1,t_2)=\max(2t_2-1, 2t_1 + t_2-1).
\end{equation}

Burr presented two colourings that show his conjectured bound~\eqref{eq:burr} is a lower bound on the Ramsey number of every $(t_1,t_2)$-tree $T$. The first colouring consists of two disjoint blue cliques of size $t_2-1$, with all edges in between coloured red. This clearly does not contain a blue copy of $T$, as it has no monochromatic blue component of size $t$. 
It also does not contain a red copy of $T$, because the red subgraph is bipartite and neither colour class has $t_2$ vertices.  Burr's second colouring is the disjoint union of a blue clique of size $t_1-1$ and a blue clique of size $t_2+t_1-1$, with all crossing edges again coloured red. This clearly does not have a blue copy of $T$, because the blue components are too small. There is also no red copy of $T$, as the red subgraph is bipartite with one part having fewer than $\min(t_1,t_2)$ vertices. 
Thus ${\cal R}(T) \ge \min(2t_2-1,2t_1+t_2-1)$. 

The reader may have observed that simply considering Harary's observation about stars above, 
Burr's conjecture \eqref{eq:burr} is off by an additive factor of 1 for any $(t_1,t_2)$ with $t_2 \ge 2t_1$. In 1979, Grossman, Harary and Klawe \cite{ghk79} extended Harary's upper bound to the class of {\it double stars}, which are trees obtained by adding an edge between the roots
of two stars.\footnote{We note that a star is a double star where one substar is a singleton.}   They showed that provided $t_2 \ge 3t_1-2$, the conjectured bound~\eqref{eq:burr} is tight for a double star if $t$ is even, and off by one if $t$ is odd.

In 1982, Erd\H{o}s, Faudree, Rousseau, and Schelp~\cite{EFRS} asked whether Burr's bound was correct for $t_2=2t_1$. In 2016, Norin, Sun, and Zhao~\cite{NSZ} disproved Burr's conjecture in this case by  showing that for $t_2=2t_1$ the Ramsey number of a double star may be as high as $4.2t_1$,  5\% greater than Burr's conjectured bound.\footnote{In \cite{DS}, Dubo and Stein   showed this Ramsey number is no greater than $4.275t_1$.} Their results also showed that Burr's bound was off by a constant multiplicative factor for $t_2=\eta t_1$ whenever $\frac{7}{4} < \eta < \frac{105}{41}$. 

Two attempts to rescue Burr's conjecture on which we focus are restricting the ratio $t_2/t_1$, and restricting the maximum degree of $T$.\footnote{We note that the maximum degree of a double star is $t_2$, which is the maximum possible degree of a $(t_1,t_2)$-tree.} We will also be interested in determining just how far the Ramsey number can be from Burr's conjectured bound~\eqref{eq:burr} given these types of restrictions.

The result of Norin et al. left open the possibility that Burr's conjecture is true up to a constant additive error when $t_2 \ge \frac{105t_1}{41}$ or $t_2 \le \frac{7t_1}{4}$. One of our results is to show that this is not true for $t_2 \ge \frac{105t_1}{41}$ unless $t_1=O(1)$. Moreover, we determine the error term in terms of $t_1$ up to a multiplicative factor. 

\begin{theorem}
 \label{maintheorem}
    For any $t_2 \ge 2t_1$, we have $f(t_1, t_2) = 2t_2 + \Theta \left( \max \left( \frac{t_1^2}{t_2}, \sqrt{t_1} \right) \right)$.
\end{theorem}

One feature common to these counterexamples to Burr's conjecture --- both the double stars and our own constructions --- is that they feature vertices of very high degree. In particular, we have vertices of degree $t_2$. On the other hand, it has been proved that Burr's conjecture is valid for trees of bounded maximum degree. The first such result is due to  Haxell, {\L}uczak, and Tingley~\cite{HLT}, who showed that for every $\epsilon>0$ there is a $c$ such that if $T$ is a $(t_1,t_2)$-tree with maximum degree at most $ct_2$, then ${\cal R}(T) \le (1+\epsilon)\max(2t_2,t_2+2t_1)$.

Building on this work, Montgomery, Pavez-Sign\'{e}, and Yun~\cite{MPY} answered a question of Stein~\cite{S} and proved Burr's bound was tight for large enough trees of bounded degree. That is, they showed that there is a $c>0$ such that if $T$ is a sufficiently large $(t_1,t_2)$-vertex tree with maximum degree at most $c(t_1+t_2)$, then ${\cal R}(T)=\max(2t_2-1,t_2+2t_1-1)$, as Burr conjectured.

Their proofs make use of the Szemer\'edi Regularity Lemma, and as a result the constant $c$ is necessarily very small. It is natural to ask what the largest $\Delta$ is such that Burr's conjecture holds for all trees of maximum degree $\Delta$. In our next result, we essentially determine this threshold for very lopsided trees, surprisingly showing that the maximum degree can be as large as $t_2 - t_1$.

\begin{theorem}
 \label{maintheorem1}
 If $t_2 \ge 500 t_1$, then for every $(t_1, t_2)$-tree $T$ satisfying $\Delta(T) \le t_2 - t_1$, we have ${\cal R}(T) = 2t_2 - 1$. 
 On the other hand, given any $t_2 \ge 2t_1$, there is a $(t_1, t_2)$-tree $T$ with $\Delta(T) \le t_2 - t_1 +\tilde{O}(t_2^{2/3})$ and $\mathcal{R}(T) \ge 2t_2 + \Omega( \log t_2)$.
 \end{theorem}

We prove our lower bounds on the Ramsey number of ($t_1,t_2)$-trees in Section~\ref{seclb}, our upper bounds in Section~\ref{secub}, and then make some concluding remarks in Section~\ref{seccr}.
 
\section{The Lower Bounds}
\label{seclb} 
In this section, we will prove our lower bounds on the Ramsey numbers of trees, which in particular serve as counterexamples to Burr's conjecture. The following result provides the lower bound in Theorem~\ref{maintheorem}.

\begin{prop} \label{prop:main-lower-bound}
     For every sufficiently large $t_1$ and $t_2$ with  $t_2 \ge 2t_1$, there exists a $(t_1,t_2)$-tree $T$ with 
    ${\mathcal R}(T) \ge 2 t_2 + \max\left(\frac{2 t_1^2}{25 t_2},\sqrt{t_1} - 5 \right)$.
\end{prop}

This next proposition gives the lower bound in Theorem~\ref{maintheorem1}.

\begin{prop} \label{prop:bounded-degree-lower-bound}
    Let $t_2 \ge 2 t_1$ be sufficiently large, and suppose $40 t_2^{2/3} \log^{1/3} t_2 \le x \le t_1$. Then there is a $(t_1, t_2)$-tree $T$ with $\Delta(T) \le t_2 - t_1 + x$ and $\mathcal{R}(T) \ge 2t_2 + \frac{x^3}{320 t_2^2} - 4$.
\end{prop}

In particular, for any $\varepsilon > 0$ there is a $\delta > 0$ such that, for sufficiently large $t_1$ and $t_2$ with $t_1 \ge (1 + 2\varepsilon) \varepsilon t_2$, there are $(t_1,t_2)$-trees $T$ with maximum degree $(1 + \varepsilon) t_2 - t_1$ for which Burr's conjecture is off by at least $\delta t_2$.

To prove these results, we shall construct a class of trees, and then colour the edges of large enough complete graphs so as to avoid monochromatic copies of our trees.

\subsection{The Trees}

The trees we use will be variations on the stars and double-stars that have previously been studied in the context of Burr's conjecture. To simplify their description, we will refer to $K_{1,\ell}$, the star with $\ell$ leaves, as an $\ell$-star.

\begin{dfn}[Deranged stars] \label{dfn:deranged-stars}
    Given $\ell \ge m \ge 0$ and $k \ge 0$, an \emph{$(\ell, m, k)$-deranged star} is formed by taking an $\ell$-star, and attaching $k$ leaves to each of $m$ neighbours of the root vertex.
\end{dfn}

Note that, in particular, if $m = 0$ or $k=0$, an $(\ell, m, k)$-deranged star is just an $\ell$-star, while setting $m=1$ instead yields the double star mentioned in the introduction. Although these deranged stars already suffice to prove many of our lower bounds, the root vertex has large degree --- it is adjacent to all $\ell$ vertices from the other part in the tree's bipartition. To prove lower bounds with control over the maximum degree, we introduce a wider class of trees.

\begin{dfn} \label{dfn:trees}
    Given $\ell \ge m \ge 0$, $k \ge 0$, and $s \ge 0$, the tree $T(\ell, m, k, s)$ is obtained by taking an $(\ell, m, k)$-deranged star with root $r_1$, and a disjoint $s$-star with root $r_2$, and adding a path of length two between $r_1$ and $r_2$ through a new vertex.
\end{dfn}

This abundance of parameters allows us to control the ratio between the sizes of the parts of the tree, as well as the maximum degree. Indeed, observe that $T(\ell, m, k, s)$ is an $(mk + 2, \ell + s + 1)$-tree whose maximum degree is $\max( \ell + 1, s + 1, k + 1, 2)$.

\subsection{The Colourings}

Now that we have constructed our trees, we turn to the colourings of the complete graph. Our colourings will be similar to Burr's conjectured extremal colourings, in that we shall have two disjoint blue cliques. However, to improve the bounds, we shall allow a few blue edges between the cliques, with restrictions on the degrees of these crossing edges.

\begin{dfn}[$(a,b;\delta, \Delta, \gamma)$-colourings] \label{dfn:colourings}
Given integers $a, b, \delta, \Delta, \gamma$, an \emph{$(a,b;\delta, \Delta, \gamma)$-colouring} is a red-/blue-colouring of the edges of $K_{2(a+b)}$ such that:
\begin{enumerate}
    \item the vertices $V(K_{2(a+b)})$ are partitioned into two \emph{parts} $C_1, C_2$, each of size $a + b$, and each $C_i$ is further partitioned into sets $A_i, B_i$ of sizes $a$ and $b$ respectively,
    \item each $C_i$ induces a blue clique,
    \item every blue edge between $C_1$ and $C_2$ has at least one endpoint in $A_1 \cup A_2$,
    \item every vertex has between $\delta$ and $\Delta$ blue neighbours in the opposite part, and
    \item every two vertices in a part have at least $\gamma$ common blue neighbours in the opposite part.
\end{enumerate}
\end{dfn}

In the next subsection, we shall show that these colourings avoid monochromatic copies of the trees previously defined. Before we proceed, though, we construct families of $(a,b;\delta, \Delta, \gamma)$-colourings for appropriate choices of the parameters.

\begin{lemma} \label{lem:existence-of-colourings}
    For every sufficiently large a, the following colourings exist.
    \begin{itemize}
        \item[(a)] For every $0 \le d \le a$, there is an $(a,0;d,d,0)$-colouring of $K_{2a}$.
        \item[(b)] Provided $b \ge \max( a, \Delta)$, there is an $(a, b; \left \lfloor \frac{a \Delta}{b} \right \rfloor, \Delta, 0)$-colouring of $K_{2(a+b)}$.
        \item[(c)] For every $0 < p < 1$, and $a,b$ such that $ap^2 > 17 \log (a+b)$, there is an $(a, b; \frac12 ap, 2(a+b)p, \left \lfloor \frac12 p^2 a \right \rfloor )$-colouring of $K_{2(a+b)}$.
    \end{itemize}
\end{lemma}

\begin{proof}
    \begin{itemize}
        \item[(a)] We take $B_1$ and $B_2$ to be empty, so that $A_1 = C_1$ and $A_2 = C_2$ are disjoint blue cliques of order $a$. We embed any $d$-regular bipartite subgraph as the blue edges between these cliques, with all other edges between the cliques coloured red.
        \item[(b)] Let $\Gamma$ be a bipartite graph with parts $X$ of size $a$ and $Y$ of size $b$, such that every vertex in $X$ has degree $\Delta$, while every vertex in $Y$ has degree $\left \lfloor \frac{a\Delta}{b} \right\rfloor$ or $\left \lceil \frac{a \Delta}{b} \right \rceil$. We then embed two blue copies of $\Gamma$ between the blue cliques $C_1$ and $C_2$: one with $A_1 = X$ and $B_2 = Y$, and the other with $A_2 = X$ and $B_1 = Y$. All other edges between $C_1$ and $C_2$ are coloured red.
        \item[(c)] We colour every edge between the blue cliques $C_1$ and $C_2$ that has an endpoint in $A_1 \cup A_2$ blue with probability $p$, independently of all other edges.

        For $i = 1, 2$, the blue-degree in $C_{3-i}$ of a vertex in $A_i$ has a binomial distribution $\mathrm{Bin}(a+b,p)$, while the blue-degree in $C_{3-i}$ of a vertex in $B_i$ follows the distribution $\mathrm{Bin}(a,p)$. Furthermore, for every pair of vertices in $C_i$, the number of common blue neighbours in $A_{3-i}$ is distributed as $\mathrm{Bin}(a,p^2)$.

        By McDiarmid's Chernoff-type concentration inequalities~\cite[Theorem 2.3]{mcdiarmid1998concentration}, it follows that if $S_n \sim \mathrm{Bin}(n,q)$, then $\mathbb{P}(S_n \le \frac12 nq)$ and $\mathbb{P}(S_n \ge 2nq)$ can both be bounded by $e^{-\frac18 nq}$. Using this, together with a union bound over the $2(a+b)$ vertices in the graph and $2 \binom{a+b}{2}$ pairs of vertices within the cliques, we deduce that with positive probability, each vertex has between $\frac12 ap$ and $2(a+b)p$ blue neighbours in the other clique, while every pair of vertices has at least $\frac12 ap^2$ common blue neighbours in the opposite part. \qedhere
    \end{itemize}
\end{proof}

\subsection{No Monochromatic Trees}

We now show that, for appropriate choices of the parameters, the $(a,b;\delta, \Delta, \gamma)$-colourings do not contain monochromatic copies of the deranged stars and the $T(\ell,m,k,s)$ trees. We start by considering the red subgraph.

\begin{lemma} \label{lem:no-red-copy}
    Suppose, for some parameters $a, b, \delta, \Delta, \gamma \ge 0$, we have an $(a,b;\delta, \Delta, \gamma)$-colouring of $K_{2(a+b)}$. If $\ell > a + b - \delta$, then there is no red $(\ell, m, k)$-deranged star. If $\ell + s \ge a + b - \gamma$, then there is no red copy of $T(\ell, m, k, s)$.
\end{lemma}

\begin{proof}
    First note that in an $(a,b;\delta, \Delta, \gamma)$-colouring, the red edges form a bipartite subgraph with parts $C_1$ and $C_2$. Furthermore, since every vertex is incident to at least $\delta$ blue edges between the parts, it follows that the maximum degree of the red subgraph is at most $a + b - \delta$. Hence, if $\ell > a + b - \delta$, then there is no red $(\ell, m, k)$-deranged star, as the root of such a star has degree $\ell$.

    Next consider the case when $\ell + s \ge a + b - \gamma$, and suppose for contradiction we had a red copy of $T(\ell, m, k, s)$. As the red subgraph is bipartite, the images of the roots $r_1$ and $r_2$, say $v_1$ and $v_2$, must lie in the same part of the colouring, say $C_1$. As $v_1$ and $v_2$ have at least $\gamma$ common blue neighbours in $C_2$, it follows that the union of their red neighbourhoods has size at most $a + b - \gamma$. However, in $T(\ell, m, k, s)$, there are $\ell + s + 1$ vertices adjacent to $r_1$ or $r_2$, and thus it follows that $v_1$ and $v_2$ do not have enough red neighbours in which to embed a copy of $T(\ell, m, k, s)$.
\end{proof}

We now turn to the blue subgraph of an $(a,b;\delta, \Delta, \gamma)$-colouring, showing that, provided the parameters satisfy the appropriate conditions, this will also be free of $(\ell,m,k)$-deranged stars.

\begin{lemma} \label{lem:no-blue-copy}
    Provided $\ell + mk \ge a + b + ( m + k + 1) \Delta$ or $\ell + mk \ge a + b + (k+1)(a + \Delta)$, an $(a,b;\delta, \Delta, \gamma)$-colouring does not contain a blue $(\ell,m,k)$-deranged star.
\end{lemma}

\begin{proof}
    Suppose for contradiction that there was a blue copy of an $(\ell, m, k)$-deranged star in an $(a,b;\delta, \Delta, \gamma)$-colouring. Let $v$ be the image of the root $r$ of the star, and suppose without loss of generality that $v \in C_1$. We shall, in two different ways, bound the number of vertices of the deranged star that can come from $C_2$, and conclude that there is not enough space to embed the deranged star within the blue subgraph.

    First, observe that in the deranged star, $r$ has $m$ neighbours that each have $k$ children of their own, with the remaining $\ell - m$ neighbours of $r$ being leaves. The vertex $v$ has at most $\Delta$ blue neighbours in $C_2$, and for each of these neighbours we can embed $k$ children in $C_2$. The remaining neighbours of $r$ must be embedded within $C_1$. Each such vertex has at most $\Delta$ blue neighbours in $C_2$, and so for the at most $m$ remaining non-leaf neighbours of $r$, we can embed at most $\Delta$ of their children in $C_2$. In total, this blue copy of the $(\ell, m, k)$-deranged star can have at most $(m + k + 1)\Delta$ vertices from $C_2$, together with the $a + b$ vertices of $C_1$. Thus, if the first inequality holds, then we do not have enough space to embed the $\ell + mk + 1$ vertices of the $(\ell, m, k)$-deranged star.

    We can also bound the vertices embedded in $C_2$ through a different argument. As before, $v$ has at most $\Delta$ blue neighbours in $C_2$, and for each of these we can embed a further $k$ children in $C_2$. We can also embed neighbours of $r$ in $A_1$, each of which could have its $k$ children embedded in $C_2$, allowing us to utilise a further $a k$ vertices of $C_2$ this way. Finally, the remaining neighbours of $r$ must be embedded in $B_1$. However, all of the blue edges between $B_1$ and $C_2$ have their endpoints in $A_2$, and thus we can reach at most another $a$ vertices of $C_2$ through these neighbours. In total, this gives us at most $(k+1)(a + \Delta)$ vertices of $C_2$, together with the $a + b$ vertices of $C_1$. Thus, if the second inequality holds, then we again do not have enough space to embed the $\ell + mk + 1$ vertices of the $(\ell, m, k)$-deranged star.
\end{proof}

\subsection{Deducing Our Results}

We can now combine the above results to obtain the lower bounds on the Ramsey numbers of unbalanced trees. We begin by proving Proposition~\ref{prop:main-lower-bound}, showing that the Ramsey number can be considerably larger than the bound conjectured by Burr.

\begin{proof}[Proof of Proposition~\ref{prop:main-lower-bound}]
    Note that if the first term in the maximum dominates, we have $t_1 = \Omega(t_2^{2/3})$. To establish this bound, we let $T$ be an $(\ell, m, k)$-deranged star with $\ell = t_2$, $m = t_1 - 1$, and $k=1$. For the colouring, we set $a = \frac{t_1}{5}$, $b = t_2 - \frac{t_1}{5} + \frac{t_1^2}{25 t_2}$, and $\Delta = \frac{t_1}{5}$, and then take an $(a,b; \left \lfloor \frac{a\Delta}{b} \right\rfloor, \Delta, 0)$-colouring, which  exists by part (b) of Lemma~\ref{lem:existence-of-colourings}. We then have $a + b - \left \lfloor \frac{a\Delta}{b} \right\rfloor = t_2 + \frac{t_1^2}{25t_2} - \left \lfloor \frac{t_1^2}{25 \left(t_2 - \frac{t_1}{5} + \frac{t_1^2}{25t_2} \right)} \right\rfloor < t_2$, and so $\ell > a + b - \delta$, which by Lemma~\ref{lem:no-red-copy} implies we have no red copy of $T$. On the other hand, $a + b + 2(a + \Delta) = t_2 + \left( \frac45 + \frac{t_1}{25 t_2} \right) t_1 < t_2 + t_1$. We thus have $\ell + mk \ge a + b + (k+1)(a + \Delta)$, and so by Lemma~\ref{lem:no-blue-copy}, it follows that there is no blue copy of $T$ either. Hence, as we have no monochromatic copy of $T$, it follows that $\mathcal{R}(T) \ge 2(a+b) = 2t_2 + \frac{2 t_1^2}{25 t_2} $.

    For the second bound, let $T$ be an $(\ell, k, k)$-deranged star, where $k$ is even and $\ell \ge 2k^2 + 2$. We then have $t_2 = \ell$, and $t_1 = k^2 + 1$. We take an $(a,b; \delta, \Delta, \gamma)$-colouring with $a = \ell + \frac12 k - 2$, $b = 0$, $\delta = \Delta = \frac12 k - 1$, and $\gamma = 0$, which exists by part (a) of Lemma~\ref{lem:existence-of-colourings}. We have $a + b - \delta = \ell + \frac12 k - 2 - ( \frac12 k - 1) = \ell - 1 < \ell$, and so by Lemma~\ref{lem:no-red-copy}, there is no red copy of $T$. On the other hand, $a + b + (m + k + 1)\Delta = \ell + \frac12 k - 2 + (2k + 1) (\frac12 k - 1) = \ell + k^2 - k - 3 < \ell + k^2 = \ell + mk$, and so by Lemma~\ref{lem:no-blue-copy}, there is no blue copy of $T$ either. Thus, this colouring has no monochromatic copy of $T$, showing $\mathcal{R}(T) \ge 2(a + b) = 2 \ell + k - 4 > 2 t_2 + \sqrt{t_1} - 5$.
\end{proof}

We conclude this section by proving Proposition~\ref{prop:bounded-degree-lower-bound}, showing Burr's conjecture can fail for trees with somewhat smaller maximum degrees.

\begin{proof}[Proof of Proposition~\ref{prop:bounded-degree-lower-bound}]
    We take $T = T(\ell, m, k, s)$, where $\ell = t_2 - t_1 + x - 1$, $m = t_1 - 2$, $k = 1$, and $s = t_1 - x$. With these parameters, $T$ is a $(t_1,t_2)$-tree with $\Delta(T) = t_2 - t_1 + x - 1$.

    Next, set $p = \frac{x}{8 t_2}$, and take $a = \frac{x}{5}$ and $b = t_2 - a + \frac12ap^2 - 2$. Since $ap^2 = \frac{x^3}{320t_2^2} \ge 200 \log t_2 > 17\log (a + b)$, we may apply Lemma~\ref{lem:existence-of-colourings}(c) to obtain an $(a,b; \frac12ap, 2(a+b)p, \left \lfloor \frac12ap^2 \right \rfloor)$-colouring of $K_{2(a+b)}$.

    We have $a + b - \gamma = t_2 + \frac12 a p^2 - 2 - \left \lfloor \frac12 ap^2 \right \rfloor \le t_2 - 1 = \ell + s$, and so it follows from Lemma~\ref{lem:no-red-copy} that there is no red copy of $T$.

    Meanwhile, we have $a + b + (k+1)(a + \Delta) \le t_2 + \frac12 ap^2 + 2 \left( a + 2(t_2 + \frac12 ap^2)p \right) = t_2 + \frac{x^3}{640 t_2^2} + \frac{2x}{5} + \frac{x}{2} + \frac{x^4}{1280t_2^3} < \ell + mk = t_2 + x - 3$, and so Lemma~\ref{lem:no-blue-copy} shows that we do not have a blue copy of $T$ either.

    Thus, $\mathcal{R}(T) > 2(a + b) = 2 t_2 + ap^2 - 4 = 2t_2 + \frac{x^3}{320 t_2^2} - 4$. 
\end{proof}

\section{The Upper Bounds}
\label{secub} 
In this section, we prove our upper bounds on the Ramsey Number of $(t_1,t_2)$-trees satisfying
$t_2 \ge 500t_1$. In particular, the following result provides the upper  bound in Theorem~\ref{maintheorem}.\footnote{We easily have $f(t_1, t_2) = O(t_2)$, as we can greedily embed $T$ in a monochromatic subgraph of sufficiently large minimum degree. This settles the upper bound when $t_1$ and $t_2$ are comparable, and we are thus free to assume $t_2 \ge 500 t_1$.}

\begin{prop}\label{mainlem}
If   $t_2 \ge 500t_1$ and $n=2t_2 + \max\left(\frac{5t_1^2}{2t_2}, 20\sqrt{t_1}\right)$  then, for any colouring $G_r\cup G_b$ of $K_n$, every $(t_1,t_2)$-tree $T$ is a subgraph of $G_r$ or $G_b$. 
\end{prop}

 The following result provides the upper  bound in Theorem~\ref{maintheorem1}.
 
\begin{prop}\label{mainlem3}
  If $t_2 \ge 500 t_1$, and $n=2t_2-1$ then  every $(t_1, t_2)$-tree $T$ satisfying $\Delta(T) \le t_2 - t_1$ is a subgraph of $G_r$ or $G_b$.
\end{prop} 

The hypotheses of both these propositions imply that  $t_1=\rho t_2$ for some $\rho \le  \frac{1}{500}$, and the given number of vertices then satisfies $n \le 2.01 t_2$. We also have $2t_2-1 \le n$, with equality implying $\Delta(T) < t_2$. 

Suppose we are given $n$ and a $(t_1,t_2)$-tree $T$ satisfying these common conditions. We will first show (see Lemma~\ref{commonlem}) that if a $2$-colouring of the edges of $K_n$ does not yield a monochromatic copy of $T$, then it must be close in structure to the $(a,b;\delta, \Delta, \gamma)$-colourings of the previous section. That is, up to renaming the colours, the vertices can be partitioned into two cliques that are nearly all blue, with almost all the edges between the cliques coloured red.

We shall then complete our proofs by showing that, with the specific assumptions of each proposition, the red graph $G_r$ or the blue graph $G_b$ must contain $T$ as a subgraph.



\subsection{Some Common Machinery} 

Given a vertex $x \in V(T)$ and a component $K$ of $T - x$, the \emph{root} of $K$ is the neighbour of $x$ it contains. 
We let $s$ be a vertex of $T$ that minimizes the size of the largest component of $T-s$. Note that such a largest component $K$ has size at most $\frac{t_1 + t_2}{2}$, as otherwise the root $s'$ of $K$ would contradict the choice of $s$ (since every component of $T-s'$ is either $T-K$ or strictly contained in $K$). 
We root $T$ at $s$. 

We let $L_2$ be the leaves in $I_2$ and note that every vertex of $I_2-L_2$ has a child in $I_1$. So $I_2-L_2$ has at most $t_1$ vertices, and hence $|L_2| \ge t_2-t_1$ and $|V(T)-L_2| \le 2t_1$. 

When finding our copies of $T$ in $G_b$ or $G_r$, we will embed the relatively small subtree $T - L_2$, possibly with some of $L_2$ as well, in vertices of high degree. It will then remain to embed the remaining leaves of $L_2$, for which we will make use of the following well-known result on systems of distinct representatives.

\begin{theorem}(Hall's Theorem~\cite{Hl})
For any family ${\cal F}$ of subsets of a ground set $X$, there are distinct $x_F \in F$ for each $F$ in ${\cal F}$ if and only if there is no ${\cal F}' \subseteq {\cal F}$ such that
\[ |\cup_{F \in {\cal F}'} F| <|{\cal F}'|. \]
\end{theorem}

In our setting, this yields the following corollary.

\begin{cor} \label{cor:Hall}
For any subset $L$ of $L_2$, given an embedding $\phi$ of $T-L$ into a graph $G$, we can extend $\phi$ to an embedding of $T$ if and only if there is no subset $L'$ of $L$ such that the union $U$ of the neighbourhoods in $G$ of the images of the parents of the leaves in $L'$ has size less than $|L'|+|\phi^{-1}(U)|$. 
\end{cor}

For an immediate application of this corollary, observe that if we embed $T - L_2$ in such a way that the images of the vertices in $I_1$ have degree at least $t_2 + t_1$, then the condition is immediately satisfied, allowing us to complete the embedding. 
This implies that we cannot have too many vertices of high degree in any given colour class.

\begin{obs}
    If there is no copy of $T$ in $G_r$, then $G_r$
    has fewer than $2 t_1$ vertices of degree at least $\frac{n}{2}+2 t_1$. 
\end{obs}

\begin{proof}
We assume the contrary  and embed $T$ for a contradiction. Let $S$ be a set of $2t_1$ vertices of degree at least $\frac{n}{2} + 2t_1$, and observe that the bipartite subgraph of $G_r[S, V-S]$ has average degree at least $2t_1$, and hence contains a subgraph of minimum degree $t_1$. We can greedily embed $T-L_2$ in this subgraph, placing $I_1$ in $S$, and then embed $L_2$ by Corollary~\ref{cor:Hall}.
\end{proof}

Applying the same argument in $G_b$, we obtain:  

\begin{obs} If there is no monochromatic copy of $T$ then there a set $Z$ of at most $4 t_1-2$ vertices such that each of $G_r-Z$, $G_b-Z$ has minimum degree at least $\frac{n}{2}-6t_1+2$ and maximum degree at most $\frac{n}{2}+2t_1$.    
\end{obs}
 
We often restrict our embedding of $T-L_2$ to $V-Z$, which makes the applications of Hall's theorem easier as its vertices have degree at least $t_2-6t_1$ in both $G_r$ and $G_b$. How we carry out this embedding depends on the shape of $T$. We call a vertex $c$ of $T$ a \emph{center} if it is incident to more than $\frac{t_2}{4}$ leaves, noting that any such center must belong to $I_1$. We will typically first embed $c$ in some specially chosen vertex $v$ of degree at least $t_2$ in $G_r$ or $G_b$, and then attempt to embed $T-(L_2 \cap N(c))$ so that $t_1$ of its vertices lie outside $N(v)$. This allows us to embed $L_2 \cap N(c)$ arbitrarily in the unused neighbours of $v$.

On the other hand, if $T$ has no center, then we exploit this fact in two ways.

First, we order the vertices of $I_1 = \{s_1, s_2, \hdots, s_{t_1} \}$ by the number of leaf children they have, and adaptively embed these vertices one-by-one.
By doing so carefully, we will be able to apply Hall's Theorem to complete the embedding. 
In particular, we will use the lack of a center to deduce that for any $S \subseteq I_1$ containing at least one element from each pair in $\{(s_{2i-1},s_{2i}) : 2i \le |I_1|\}$, letting $\ell_i$ be the number of leaf children of $s_i$, we have $\sum_{j : s_j \in S} \ell_j > \tfrac12(|L_2|-\ell_1)> \tfrac{t_2}{3}$. 

Second, we note that since the root $s$ is itself not a center, it can have at most $\frac{t_2}{4}$ leaf children. This implies that the at most $t_1$ nonsingleton components of $T-s$ contain at least $\frac{3t_2}{4}$ vertices of $I_2$. Since the largest such component contains at most  $\frac{t_2+t_1}{2}$ elements of $I_2$, this means we can partition the components into two sets, with each set of components containing at least $\frac{t_2-2t_1}{4}$ elements of $I_2$.

With these preliminaries in place, we can now prove the following crucial lemma, which determines the approximate structure of a $2$-colouring of $K_n$ without monochromatic copies of $T$.

\begin{lemma}
\label{commonlem}
    If there is no monochromatic copy of $T$ then, renaming colours if needed, we can partition $V$ into $C$ and $D$ so that $G_r[C,D]$, $G_b[C]$ and $G_b[D]$ all have minimum degree at least $\frac{n}{2}-2t_1$, while $\Delta(G_b[C,D]) \le t_1$. Furthermore, if $T$ has a center then $\Delta(G_r)<t_2$. 
\end{lemma}
\begin{proof}
We distinguish two cases.
   
\noindent{\bf Case 1:} $T$ has a center $c$.

Wlog $V-Z$ contains a vertex $v$ such that the degree of $v$  in $G_b$ is at least $\lceil n-1/2 \rceil \ge t_2-1$. Indeed either $v$ has degree at least $t_2$ or $n=2t_2-1$  and the degree of $v$ and every other vertex of $V-Z$ in $G_b$ is  $t_2-1$. Furthermore, in the second case, $\Delta(T)<t_2$.
Since $v$ is in $Z$ this  degree is at most $n/2+2t_1$.   If there 
are $2.01t_1t_2$ blue  edges between $N_b(v)$ and $V-N_b(v)-Z-v$ then $G_b[V-N_b(v)-v-Z,N_b(v)]$ has average degree exceeding $2t_1$ and contains a subgraph of minimum degree at least $t_1$ and average degree at least $2t_1$. We consider a maximal such subgraph of minimum degree $t_1$ and average degree $2t_1$ and note it  has at least $2t_1$ vertices on each side and so any vertex of $N_b(v)$ not contained in it is nonadjacent to more than $t_1$ vertices of its vertices outside $N_b(v)$. We embed $c$ in $v$ and $T_2-L_2-c$ in this subgraph with $I_1-c$ outside 
$N_b(v)$. Furthermore, if $v$ has degree $t_2-1$ in $G_b$ and there is a vertex $w$  in this subgraph which is not in $N_b(v)$ and which either has a blue edge to a vertex outside the subgraph or satisfies $|N_b(w) \cup N_b(v)|> t_2+t_1-3$
then we embed a vertex of $I_1-s$, all of whose children are leaves,  in $w$ as the first step of our embedding.  Since (i) every vertex of $G_b-Z$ has degree exceeding  $3t_2/4+t_1$, 
(ii) at least  $t_2/4$ vertices of $L_2$ are incident to $c$, and (iii) no vertices of $I_1$ are embedded in any of the at
least $t_2-1$ vertices of $N_b(v)$,  we can then embed $L_2$ by applying Hall's theorem to obtain a contradiction unless $|N_b(v)|=t_2-1$
and we could not find the desired $w$.We note that since, in this case,  all the vertices in the subgraph in $V-N_b(v)-Z$ have degree $t_2-1$ in $G_b$  this implies 
that for each such vertex $y$, $|N_b(v)-N_b(y)|<t_1$. Since our subgraph contains every vertex of $N_b(v)$ which is nonadjacent  in $G_b$ to 
fewer than $t_1$+1 vertices of the subgraph which are not in $N_b(v)$ this implies that the number of vertices 
in $N_b(v)$ which are not in the subgraph is less than the number of vertices outside $N_b(v)$ which are in the subgraph.
So every vertex of $G_b-Z$ sees a vertex of the subgraph. Now if there is a vertex $w'$ of  $V-Z-N_b(v)$  which is  not in the subgraph, then since we could not choose $w$,  $w'$ sees none of the subgraph outside $N_b(v)$ and hence sees a vertex of the subgraph in  $N_b(v)$. Furthermore, since it is not in the subgraph it sees at most $t_1-1$ such vertices and hence sees fewer than $t_2-t_1-2$   vertices in $N_b(v)$. 
Thus we can  embed a vertex of $I_1-s$ with only leaf children in $w'$, its parent in a vertex of the subgraph in $N_b(v)$ and proceed as before to 
to complete the embedding of $T$.  So, either there are at most $2.01t_1t_2$ blue edges between $N_b(v)$ and $V-N_b(v)-Z$ or every 
vertex of $V-N_b(v)-Z$ has blue edges to all but $t_1$ vertices of $N_b(v)$ and hence there are fewer than 
$2.01t_1t_2$ red edges between these two sets. Swapping colours if necessary and repeatedly deleting the  vertex  of maximum  degree in 
$G_b[V-N(v)-v-Z,N(v)-Z]$ until we have deleted $\sqrt{2.01t_2t_1}$ vertices we obtain a bipartite subgraph $G_r[C',D']$ of $G_r$  of  minimum degree $\frac{n}{2}-2\sqrt{2.01t_2t_1}-6t_1$ each of whose colour classes contains at least 
$\frac{n}{2}-\sqrt{2.01t_2t_1}-6t_1$ vertices.

If any vertex $v$ satisfies $|N_r(v)| \ge t_2,|N_r(v) \cap C'| \ge t_1,|N_r(v) \cap D'| \ge t_1$,  then we  choose a set $M$ of $t_2$  neighbours of $v$ in $G_r$ containing $t_1$ vertices of both $C'$ and $D'$.
Wlog we can insist that $M \cap C' \le \frac{t_2}{2}$ and $M \cap D' \ge \frac{t_2}{3}$. We embed $c$ in $v$ and  the roots of the nonsingleton components 
of $T-c$ in $D'$. We embed  the rest of $T_2-L_2$ in $G_r[C'-M,D']$ which implies  we embed none of $I_1$ in $M$. We can 
now embed  $L_2$ by applying Hall's Theorem. 

So any vertex $v$ satisfying  $|N_r(v)| \ge t_2$ satisfies either $|N_r(v) \cap C'| <t_1$ or $|N_r(v) \cap D'| < t_1$. 
In particular , $G_r[C']$ and $G_r[D']$ have maximum degree $2\sqrt{2.01t_2t_1}+6t_1$ and hence 
$G_b[C']$ and $G_b[D']$ have minimum degree $\frac{n}{2}-3\sqrt{2.01t_2t_1}-12t_1$. Furthermore, 
any vertex satisfying  $|N_r(v)| \ge t_2$ has at least $\frac{n}{2}-\sqrt{2.01t_2t_1}-7t_1$ blue edges to one of 
$C'$ or $D'$ and at least $t_2-2\sqrt{2.01t_2t_1}-13t_1$ red edges to the other.

If any vertex $v$ satisfies $|N_b(v)| \ge \frac{n-1}{2},|N_b(v) \cap C'| \ge t_1,|N_b(v) \cap D'| \ge t_1$, then we can choose 
a set $M$ of $\min(t_2,\lceil \frac{n-1}{2}\rceil)$ neighbours of $v$ in $ G_b$ containing $t_1$ vertices of both $C'$ and $D'$.  Wlog $M$  contains at most $t_2/2$ vertices of  $C'$ and at least $t_2/3$ vertices of $D'$. We embed $c$ in $v$ and $T-L_2$ in $G[C']$ using $M$ only for the neighbours of $v$ which 
we embed first. We can then apply Hall's Theorem to embed $L_2$.

So any vertex $v$ satisfying  $|N_b(v)| \ge \frac{n-1}{2}$ satisfies either $|N_b(v) \cap C'| <t_1$ or $|N_b(v) \cap D'| < t_1$. 
Thus, any such   vertex has at least $\frac{n}{2}-\sqrt{2.01t_2t_1}-7t_1$ red  edges to one of 
$C'$ or $D'$ and at least $t_2 -2\sqrt{2.01t_2t_1}-13t_1$ blue  edges to the other. 

We let $C$ be those vertices which have more blue neighbours in $C'$ then $D'$ and $D=V-C$. 
We can now repeat the arguments of the first and third of the last four paragraphs with $C$ in place of $C'$ and $D$ in place of $D'$. 
We see that every vertex either (i) has degree at least $t_2$ in $G_r$  and at most $t_1$ red edges to the element of $(C,D)$,
containing it or (ii) has degree at least $\frac{n-1}{2}$ in $G_b$  and at most $t_1$ blue  edges to the element of $(C,D)$,
which does not contain it. We note that if  a vertex $v$  in $G-Z$ satisfies  (i) then we can embed $T-L_2$ with the center in $v$, $I_2-L_2$ in the side not containing $v$ and $I_1-v_1$ in nonneighbours of $v$ in the side containing it. We can then embed $L_2$. So, every vertex satisfies (ii) and not (i), and  letting $d$ be the maximum degree of $G_b(C,D)$ we have  $d \le t_1$ and that  
$\frac{n}{2}-d \le \min(G_b(C),G_b(D)) \le |C|,|D| < \frac{n}{2}+d $.
\vskip0.2cm

\noindent{\bf Case 2:} $T$ has no center.

\begin{claim}
 For some  $v\in V-Z$ fewer than  $\frac{n}{2}+8t_1$ vertices   have a red edge to  more than  
$3t_1$ blue-neighbours of $v$    
\end{claim}
\begin{proof}
    We assume the contrary and embed $T$ for a contradiction. We  first embed the at most $2t_1$ nodes of $T-L_2$ in $G_r-Z$, exploiting the fact that 
    this graph has minimum degree at least $\frac{n}{2}-6t_1$ to ensure  that 
whenever we embed  the second of the pair $(s_{2i-1},s_{2i})$, we embed it in a vertex which is 
adjacent to $3t_1$ blue-neighbours of 
the vertex in which the other element of the pair  is embedded. Since we have embedded at most $2t_1$ 
vertices, by  Hall's theorem  either we can now embed 
$L$ or there is a subset $W$ of $I_1$, such that the union of the neighbourhoods of 
the images of the elements of $W$ has size at most the minimum of $t_2+t_1$ and  $2t_1$ plus  the number of leaves 
incident to elements of $W$. By our choice of vertex images, this implies that for each $i$, 
$W$ contains at most one of $s_{2i-1}$ or $s_{2i}$. Since $T$ has no center this implies that 
the number of leaves adjacent to  the elements of $W$ is at most  $\frac{2t_2}{3}$. But the image of every 
element of $W$  has degree at least $\frac{n}{2}-t_1>\frac{2t_2}{3}+2t_1$ in $G_r$. so no such $W$ 
exists and we can embed $T$.
\end{proof}

So, we can and do choose  a vertex $v$ in $V-Z$ such that there are at most  
$\frac{n}{2}+8t_1$ vertices  which have a red edge to  more than $3t_1$ blue    neighbours of $v$.  
Hence there is set  $Y$ of at least $\frac{n}{2}-12t_1$ vertices of $V-Z$ each of which has red edges to at most $3t_1$
vertices of $V-N_r(v)$ and hence has red edges to at least 
$\frac{n}{2}-5t_1$ elements of $N_r(v)$. Any  vertex in $Y \cap N_r(v)$ has a red edge to   at most 
$3t_1$ vertices  of $Y-N_r(v)$. Any vertex of $Y-N_r(v)$ has a red edge 
to all but  at most $7t_1$ vertices of $N_r(v)$ and hence of $Y \cap N_r(v)$. So, either  $|Y \cap N_r(v)|\le 14t_1$
or $Y - N_r(v) \le 6t_1$.

If $|Y \cap N_r(v)| \le 14t_1$ vertices, then $| Y-N_r(v)| \ge \frac{n}{2}-26t_1$ and $G_r[Y-N_r(v),N_r(v)]$ 
is a bipartite graph, such that 
each vertex of $Y-N_r(v)$ is nonadjacent to at most $7t_1$ vertices of $N_r(v)$. Hence we can delete
the $\sqrt{7.1t_2t_1}$ vertices of $N_r(v)$ with lowest degree in this graph, to obtain a 
a bipartite subgraph  of $G_r$  which has minimum degree exceeding $n/2-\sqrt{7.1t_2t_1}-26t_1$.

If $Y - N_r(v)$ contains at most $6t_1$ vertices,$|Y \cap N_r(v)| \ge \frac{n}{2}-18t_1$ and   $ G_b [Y \cap N_r(v),V-N_r(v)]$ 
is a bipartite graph  such that 
each vertex of $Y \cap N_r(v)$ is nonadjacent to at most $3t_1$ vertices of $V-N(v)$. Hence we can delete
the $\sqrt{3.1t_2t_1}$ vertices of $V-N_r(v)$ with lowest degree in this graph, to obtain a 
a bipartite subgraph   of $G_r$ which has minimum degree exceeding $n/2-\sqrt{3.1t_2t_1}-18t_1$

So, wlog $G_r$ contains a bipartite subgraph $G_r[ C',D']$   which has minimum degree exceeding $n/2-\sqrt{7.1t_2t_1}-26t_1$.

If any vertex  has red edges to  at least $t_1$ vertices on one sode of this subgraph and at least $t_2/3$ on the other, then we can embed $T$ in $G_b$  as follows. We embed  $s$ in this vertex,  all the leaf 
children of $s$ on the  side in which the vertex  has  more red neighbours, and the remaining neighbours of $S$ so that  for $X \in \{C',D'\}$, the total number of vertices of $I_2$ 
in the components  containing neighbours of $S$  embedded  in $X$ is at least $\frac{t_2}{4}-t_1$.
We can then finish the embedding greedily in $G_r[C',D']$. So such $s$ exists and in particular 
 $G_b[C']$ and $G_b[D']$ have minimum degree exceeding $n/2-\sqrt{7.1t_2t_1}-27t_1$. 

So,  if there is  a vertex $v$ in $V$ which has blue edges to at least $t_2/3$ vertices on one side of the 
subgraph and  at least $t_1$ vertices on the other   we can embed $T$ in $G_b$ as follows. 
We embed $s$ in this vertex, all the leaf 
children of $s$ on the  side in which the vertex  has  more blue  neighbours, and the remaining neighbours of $S$ so that  for $X \in \{C',D'\}$, the total number of vertices of $I_2$ 
in the components  containing neighbours of $S$  embedded  in $X$ is at least $\frac{t_2}{4}-t_1$.
We can then finish the embedding greedily in $G_b[C'] 
\cup G[D']$. So no such $s$ exists.

The last two paragraphs imply that every vertex has less than  $t_1$ blue edges to one of $C'$ 
and $D'$ and less than  $t_1$ red edges to the other. 
So, we can partition $V$ into $C$ and $D$ with $|C| \le |D|$ so that $G_r[C,D], G_b[C]$ and $G_b[D]$ all have minimum  degree at least  $n/2-\sqrt{3.1t_2t_1}-19t_1$.

We can now repeat the above arguments with $C$ and $D$ in the place of $C '$ and $D'$ to obtain that 
  $\Delta(G_b[C,D]), \Delta(G_r[B]),$ and $\Delta(G_r[D])$ are all less than $t_1$.  This implies $|D| \le \frac{n}{2}+t_1$, as otherwise we could embed $T$ greedily in $G_r[C,D]$ with $I_1 \in C$. So, $|C| \ge \frac{n}{2}-t_1$
and $\delta(G_b[C]), \delta(G_b[D]), \delta(G_r[C,D]) > \frac{n}{2}-2t_1$. 
\end{proof}

\subsection{The Proof of Proposition \ref{mainlem}}
We can now complete the proof of Proposition \ref{mainlem}.
To prove the proposition, we need to show that if   $t_2 \le 500t_1$ and $n=2t_2 + \max(2.5\rho t_1, 20\sqrt{t_1})$  then $T$ is a subgraph of one of $G_r$ or $G_b$.  
We assume the contrary and apply Lemma \ref{commonlem}  to obtain the partition $(C,D)$ of $V$ it guarantees exists. Without loss of generality we may assume $|C| \le |D|$, 
 so $|D| \ge t_2+ \max(1.25 \rho t_1, 10 \sqrt{t_1})$.
Again we distinguish two cases: 

{\bf Case 1: } $T$ has no center.

We recall  every vertex of $C$ (resp. $D$) has at most $t_1-1$ blue edges to $D$ (resp. $C$).
We claim we can embed $T-L_2$ in $G_r[C,D]$ with $I_1 \in C$ so that for every pair  in $\{(s_{2i-1},s_{2i}) : 2i \le |I_1|\}$,
$|(N_r(s_{2i-1}) \cup N_r(s_{2i})) \cap D| \ge t_2$. If this is true we can then embed $L_2$ by applying Hall's theorem.
So, it remains to prove the claim. We note that letting $s_j$ be the first node  of a pair embedded and $s_k$
the second it is enough to show that  the image  of $s_k$ has blue edges to at most $1.01\rho t_1$ vertices of $D-N_r(v_j)$
for the image $v_j$ of $s_j$. 

Now, $|D-N_r(v_j)| <t_1$ and each element of this set has less than $t_1$ blue edges to $C$. So the average number of blue edges from  a vertex of $C$ into this set is less than $\frac{t_1^2}{|C|}< 1.005\rho t_1$. So, more than $|C|/50$ vertices of 
$C$ have less than $1.04\rho t_1$ blue edges to this set. Since the parent of $s_k$ is embedded in a vertex of $D$ 
which has red edges to at least $\frac{248}{249}|C|$ vertices  of  $C$ and we are  embedding fewer than  $|C|/249$ vertices into $C$ we will be able to choose 
the image of $s_k$  so it has the desired property.

{\bf Case 2:} $T$ has a center $c$

We recall every vertex has degree exceeding $n-1-t_2$ in $G_b$. We note this implies that every vertex of $C$ 
has at least $\max(1.25\rho t_1, 10\sqrt{t_1})$ blue neighbours in $D$. 

We let $m<t_1$ be the number of nonsingleton components of $T-c$. If there is a vertex $v$  in one of $C$
or $D$ which is joined by $m$ blue edges to vertices in the other  then we embed $c$ in $v$, 
the nonleaf neighbours of $c$ in vertices joined to it by some of these edges, 
and the rest of $T-L_2$  in $G_b[X-N_b(v)]$ where $X$ is the element of $(C,D)$ not containing $v$. 
We can then finish off by applying Hall's theorem.  

So, every vertex of $C$ (resp. $D$) is joined by red edges to  all but at most $m<t_1$ vertices  of $D$ (resp. $C$). We embed $c$ in a vertex  $v$ of $C$. Our approach is to embed nonsingleton components of $T-c$  containing at least $t_1$ nonneighbours of $c$  in $G_b$  so that 
no nonneighbours of $c$ are embedded in $N(v)$.   We can then complete the embedding greedily if we embed the leaf neighbours of 
$c$ last. 

To this end, we consider the nonsingleton components of $T-c$ in decreasing order 
of size, embedding their roots in $N_b(v) \cap D$ until we use up this set.
If  the union of the components containing these roots contains  at least $t_1$ nonneighbours of $c$ then  we can embed these components  in $G_b[D]$ and are done.

We note that we have embedded the roots of at least $10 \sqrt{t_1}$ components of $T-c$,
so all the remaining components have size at most $\frac{\sqrt{t_1}}{10}+1$.

We will eventually  embed  in $G_b[D]$, the components whose roots have been  embedded there. First however
while  the union of these components with their roots removed and the vertices embedded in $D-N_b(v)$ is
less than $t_1$, if possible we embed the root of the component of some unembedded 
nonsingleton component of $T-c$ in $N_b(v) \cap C$ and the rest of the component in $G_b[D]$.
This will be possible unless  there is no unused vertex of $N_b(v) \cap C$  joined by blue edges 
to at least $\frac{\sqrt{t_1}}{10}$ unused vertices of $D$.

During this process the number of unembedded nonroots in the components of $T-c$  whose roots have 
been embedded in  $D$ is at least as large as the number of such roots. 
Thus the total number of blue edges from vertices of $C$ to used vertices of $D$ is less than  $t_1^2$.
Hence the average number of blue edges from a vertex of $C$ to a used vertex of $D$ is less than 
$1.005\rho t_1$ and thus for at least $\frac{|C|}{100}$ vertices of   $C$, more than  a hundredth of their
blue edges  to $D$ (and hence more than $\frac{\sqrt{t_1}}{10}$ such edges)  go to unused vertices. 
So, we will be able to embed $t_1$ vertices in $D-N_b(v)$. We can now finish the embedding in $G_b$ greedily, 
provided we embed the leaf children of $c$ last.

\subsection{The Proof of Proposition  \ref{mainlem3}}

We can now complete the proof of Proposition \ref{mainlem3}.
We need to show if    $t_2 \ge 500t_1$, then for every $(t_1,t_2)$-tree $T$ satisfying 
     $\Delta(T) \le t_2-t_1$,  and $n=2t_2-1$, $T$ is a subgraph of one of $G_r$ or $G_b$.

To prove this statement, we assume it fails  for some $T$  to obtain a contradiction. Hence, we can apply Lemma \ref{commonlem} to obtain the 
partition of $V$ into $C$ and $D$ which that  lemma ensures  exists. Wlog $|C| \le |D|$.

{\bf Case 1}: $T$ has a center $c$. 

We let $d$ be the maximum number of blue neighbours that a vertex of $C$ has in $D$ and $v$ be a vertex of $C$ with $d$ neighbours in $D$.
So, $G_b[C]$
has minimum degree $\frac{n-1}{2}-d$. If there are at most $d$ components of $T-c$ which contain 
two vertices of  $I_2$ then we embed $T$ in $G_b$ with $c$ in $v$ as follows.  By our bound on $\Delta(T)$, there are $d$ components of $T-c$ which  contain at least $2d+t_1$ vertices. We embed these components in $D$ and the remaining components in
$C$. 

So, we can assume that there are $d$ components of $T-c$ containing two vertices of $I_2$.  We  embed $T$ in $G_r[C,D]$ with $c$ in $v$ as follows. 
We embed $d$ disjoint paths from   a child to a grandchild of $c$ using the $d$ vertices of $D$ joined to $v$ by blue edges 
as images of the grandchildren. We can then embed the rest of $T$ greedily provided we embed the leaf children of $c$  last. 

{\bf Case 2}: $T$ has no center. 

In this case we consider the maximum degree $\Delta$ in $G_b[C,D]$ and let $v$ be a vertex achieving it. 
We recall $\Delta \le t_1$. 
We note that every vertex of $C$ has degree at least $|D|-\Delta \ge t_2-\Delta$ 
in $G_r[C,D]$. So, we can embed greedily in this graph unless $|D| < t_2+\Delta$ and hence $|C| \ge t_2-\Delta$. 
Thus, $\Delta>0$ and $G_r[C,D]$ has minimum degree at least $t_2-2\Delta$. 

 For a component $K$ of $T-s$, we let  $f(K)=|I_2 \cap K|-|I_1 \cap K|$.
 We note that the sum of $f(K)$ over the at most $t_1$  nonsingleton components of $T-S$ is at least $\frac{3t_2}{4}-t_1 >300t_1$
 and that no $f(K)$ exceeds $\frac{t_2+t_1}{2}$. So, we can find a set $S$  of  at most $\lceil\Delta/150\rceil$ components 
 of $T-s$ whose $f$ value is at least $2\Delta$ and at most $\frac{t_2+t_1}{2}$. If   there is a vertex $u$  which has red edges 
 to $\lceil\Delta/150\rceil$ vertices on its side, we embed $T$ in $G_r$ as follows. We embed  $s$ in $u$, the roots of the components in $S$ in some of these neighbours, and the roots of the remaining components of $T-s$ in the element of $(C,D)$ not containing $u$. We can now finish the embedding greedily in $G_r[C,D]$.

 So, we can assume that both $G_b[C]$ and $G_b[D]$ have minimum degree at least $t_2-\tfrac{151}{150}\Delta$.
 Thus, if there are $\Delta$ components of $T-s$ whose union has size at least $t_1+\tfrac{151}{150} \Delta$ then we can
 choose a subset whose union in addition has  size at most $\frac{t_2+t_1}{2}$ and 
 embed $T$ in  $G_b$ with $T-s$ in $G_b[C] \cup G_b[D]$. 

We note that since the average nonsingleton component of $T-s$ has size exceeding $\frac{3t_2}{4t_1} \ge 302$, this implies $t_1 \ge 300$
and so $t_2 \ge 150000$. Also, the $\Delta$ largest components of $T-s$ have size less than $3t_1$. 

We claim that if  we can embed $T-L_2$ in $G_r[C,D]$ with $I_1 \subseteq C$ so that every vertex of $D$ is a neighbour in this 
graph of the image of the parent of at least $\Delta$ leaves then $T$ is a subgraph of $G_r$.  To prove the claim we assume the contrary for a contradiction. Thus there must be a nonempty subset $L'$ of $L_2$ such that the union of the neighbourhoods of the images of the parents of the elements of $L'$ 
and the leaves in $L_2-L'$ has size less than   $t_2$.  Since every vertex of $C$ has more than $t_2-\Delta$ neighbours in 
$D$, it follows that $|L_2-L'|<\Delta$. Thus, every vertex of $D$ is in the neighbourhood of the image of the parent of  some element of $L'$ and since $|D| \ge t_2$ 
we obtain a contradiction. 

We consider a random process for embedding $T-L_2$ in $G_r[C,D]$ with $I_1 \subseteq C$. We embed a random uniformly chosen vertex 
of $I_1$ in  a random uniformly chosen vertex of $C$ and then embed $T-L_2$ in top down fashion always embedding a node in a random 
uniformly chosen unused neighbour of its parent.  For 
each vertex $v$ in $D$ we let $f(v)$ be the sum of the number of leaf children of the neighbours of $v$ in $G_r[C,D]$. 
By our claim, if we can show that with positive probability every $f(v)$ is at least $t_1$ then there is a copy of 
$T$ in $G_r$. By the union bound, we need only show that for every $v \in D$, the probability that $f(v)<t_1$ is at most $\tfrac{1}{2t_2}$. 

To this end, we focus on the set $P={p_1,...p_r}$ of vertices of $I_1$ which are the parents of at most 
$\frac{t_2}{5 (\log t_2)}$ leaves.  We let $\ell_i$ be the number of leaves which are children of $p_i$. We consider the sum $S_P$ of a family of independent random variables, $\{x_1,...,x_r\}$ where  $x_i$  is  $0$ with probability $1/10$ and otherwise is  $\ell_i$.

When we come to embed $p_i$, regardless of our choices so far, there are between  $t_2-3t_1$  and $t_2$ choices of 
vertices of $C$ in which to embed it and $v$ is adjacent in $G_r[C,D]$ to at least $t_2-4t_1$ of these choices. So the
probability  that $v$ is adjacent to the image of $p$ is much higher than $\frac{9}{10}$. Hence we need only show that the 
probability that the $x_i$ sum to  less than $t_1$ is at most $\frac{1}{2t_2}$.
To bound this probability, we apply Hoeffding's Inequality \cite[Theorem 2]{hoef} which implies the following: 

For any random variable $S=X_1+...+X_n$ where the $X_i$ are independent real valued  and  $a_i \le X \le b_i$,  
\begin{equation*}
\mathbb{P}\big(|S-\mathbb{E}(S)| > y \big)
\le 
2\exp\left({\frac{-2y^2}{\sum_{i=1}^n(b_i-a_i)^2}}\right).
\end{equation*}

This yields: for any random variable $S=X_1+...+X_n$ where the $X_i$ are independent nonnegative real valued  and  $X_i \le b_i \le b$,  
$$
\mathbb{P}\big(|S-\mathbb{E}(S)| > y \big)
\le 
2\exp\left({\frac{-2y^2}{b\sum_{i=1}^nb_i}}\right).
$$

We apply this with $S=S_P$, $b_i=\ell_i$, $b=\frac{t_2}{5 \log t_2}$, $y=\mathbb{E}(S_P)-t_1=\frac{9}{10}{\sum_{i=1}^r\ell_i}-t_1$. 

If $\sum_{i=1}^r \ell_i \ge \frac{t_2}{2}$  then $y>\tfrac{4}{5}\sum_{i=1}^r \ell_i$ and $b< \tfrac{2}{5\log t_2}\sum_{i=1}^r \ell_i$,
so the probability we are bounding  is less than $\frac{1}{2t_2}$.  So, we are done unless $\sum_{i=1}^r \ell_i < {t_2}/{2}$.
It follows that the $5 \log t_2$ largest components of $T-s$ contain at least $\frac{t_2}4-2t_1$ vertices.
So we have:  (i) $\Delta<5 \log t_2$,  and (ii) some node $s^*$ of $I_1$ has $\frac{t_2}{5 \log t_2}$ children.

Since $t_2 \ge 150,000$, $s^*$ has more than $\Delta$ children. It follows that if we embed $T-L_2$ in $G_r[C,D]$ and
are unable to extend this embedding to all of $T$  then the set $L'$ which demonstrates this must contain the leaf children of $s^*$. 

We embed $T-L_2$  with $I_1 \in C$ embedding $s^*$ first. 
Whenever we embed a vertex of $I_2-s^*$ we avoid the at most $\Delta^2$  vertices of $C$  which are nonneighbours in  $G_r[C,D]$ of
one of the at most $\Delta$ vertices of $C$ joined to the image of $s^*$ by a blue edge. It follows that $L'$ can only contain the leaf children  of $s^*$ which is impossible.

\section{Concluding Remarks}
\label{seccr}

In this paper, we studied to what extent Burr's conjecture holds. In our first result, Theorem~\ref{maintheorem}, we showed that for every $t_2 \ge 2t_1$, there are $(t_1, t_2)$-trees whose Ramsey number is significantly larger than conjectured, and further determined the order of magnitude of the difference between the largest possible Ramsey numbers and Burr's bound. It would be interesting to try to sharpen these results, and to determine whether the deranged stars we constructed maximise the Ramsey number.

In our second result we investigated how the maximum degree of the tree affects the validity of Burr's conjecture. The first part of Theorem~\ref{maintheorem1} asserts that when $t_2 \ge 500t_1$, Burr's conjecture holds for $(t_1,t_2)$-trees with maximum degree at most $t_2-t_1$. 
It  may be possible  to remove  the condition $t_2 \ge 500 t_1$ in this result  entirely. Even if this is impossible, $500$ is certainly  not best possible and it would be interesting to improve it to $t_2 \ge \eta t_1$ for $\eta$ significantly smaller than $500$.

The second part of Theorem~\ref{maintheorem1} shows that this bound on the maximum degree is essentially tight; for $t_2 \ge 2 t_1$, we construct $(t_1, t_2)$-trees of maximum degree $t_2 - t_1 + \tilde{O} \left( t_2^{2/3} \right)$ whose Ramsey number is significantly larger than Burr's bound, with the gap growing with the maximum degree. It remains to determine whether there are counterexamples of maximum degree $\Delta(T) = t_2 - t_1 + x$ for $x = o \left( t_2^{2/3} \right)$, or, for given $x = \tilde{\Omega}  \left( t_2^{2/3} \right)$, what the largest possible Ramsey number is.

Finally, one can ask for results in terms of the maximum degree, without fixing the sizes of the parts of the tree. As mentioned in the introduction, Montgomery, Pavez-Sign\'e, and Yan~\cite{MPY} prove that there is some constant $c > 0$ such that Burr's conjecture holds for any $t$-vertex tree $T$ with $\Delta(T) \le ct$. Our constructions when $t_2 = 2t_1$ show that we must have $c \le 1/3$, and it would be very interesting to see if this is best possible.

\bibliographystyle{habbrv}
\bibliography{Citations}

\end{document}